\documentclass[11pt,letterpaper,reqno]{amsart}

\usepackage{amsfonts}
\usepackage{amsmath}
\usepackage{amsthm, amscd}
\usepackage{mathtools}
\usepackage{mathabx}
\usepackage{amssymb}
\usepackage{mathrsfs}
\usepackage{graphicx}
\usepackage{caption}
\usepackage{subcaption}
\usepackage{float}
\usepackage{xypic}
\usepackage[abs]{overpic}
\usepackage[alphabetic]{amsrefs}
\usepackage{etex}
\usepackage{tikz-cd}
\usepackage{slashed}
\usepackage{enumitem}

\usepackage{hyperref}
\hypersetup{
	colorlinks,
	linkcolor=[rgb]{0,0,0.7},
	urlcolor=[rgb]{0,0,0.4},
	citecolor=[rgb]{0.4,0.1,0}
}

\allowdisplaybreaks

\theoremstyle{plain}\newtheorem{Theorem}{Theorem}[section]
\theoremstyle{plain}\newtheorem{Corollary}[Theorem]{Corollary}
\theoremstyle{plain}\newtheorem{Lemma}[Theorem]{Lemma}
\theoremstyle{plain}\newtheorem{Definition}[Theorem]{Definition}
\theoremstyle{plain}\newtheorem{Proposition}[Theorem]{Proposition}
\theoremstyle{plain}
\theoremstyle{plain}
\theoremstyle{plain}\newtheorem*{Claim*}{Claim}
\theoremstyle{plain}\newtheorem*{Theorem*}{Theorem}
\theoremstyle{plain}
\theoremstyle{plain}\newtheorem{Question}[Theorem]{Question}
\theoremstyle{plain}
\theoremstyle{plain}
\theoremstyle{plain}\newtheorem{lemma}[Theorem]{Lemma}
\theoremstyle{plain}
\theoremstyle{plain}
\theoremstyle{plain}

\theoremstyle{remark}\newtheorem{Remark}[Theorem]{Remark}
\theoremstyle{remark}\newtheorem{rmk}[Theorem]{Remark}
\theoremstyle{remark}
\theoremstyle{remark}
\theoremstyle{remark}\newtheorem*{Notation*}{Notation}
\numberwithin{equation}{section}
\theoremstyle{plain}

\newcommand{\Z}{\mathbb{Z}}

\newcommand{\R}{\mathbb{R}}
\newcommand{\C}{\mathbb{C}}

\newcommand{\CP}{\mathbb{CP}}
\newcommand{\RP}{\mathbb{RP}}

\newcommand{\fr}{\mathfrak}
\newcommand{\id}{\mathrm{id}}
\newcommand{\ii}{\mathrm{i}}
\newcommand{\dd}{\mathrm{d}}
\newcommand{\Cl}{\mathrm{Cl}}

\newcommand{\lc}{DT^*S^2}
\newcommand{\Diff}{\mathrm{Diff}}

\newcommand{\Aut}{\mathrm{Aut}}
\newcommand{\spinc}{\mathrm{Spin}^{\mathbb{C}}}
\newcommand{\SW}{\mathrm{SW}}
\newcommand{\FSW}{\mathrm{FSW}}
\newcommand{\ind}{\mathrm{Ind}}
\newcommand{\PD}{\mathrm{PD}}

\def\det{\mathrm{det}}

\def\dim{\mathrm{dim}}

\def\pr{\mathrm{pr}}
\def\Fr{\mathrm{Fr}}

\def\MCG{\mathrm{MCG}}

\def\cC{\mathcal{C}}

\def\cH{\mathcal{H}}

\def\cM{\mathcal{M}}

\usepackage{lipsum}

\title[Homotopy Coherent Nielsen Realization Problem]{On the Homotopy Coherent Nielsen Realization Problem for K3-Type 4-Manifolds}
\author{Yujie Lin}
\address{Qiuzhen College, Tsinghua University, China}
\email{lyj23@mails.tsinghua.edu.cn}

\author{Yi Sha}
\address{School of Mathematical Sciences, Peking University, China}
\email{Sha\_Yi\_glgjssy@pku.edu.cn}

\begin{document}

\maketitle
\begin{abstract}
    We study the homotopy coherent analogue of the Nielsen realization problem for smooth 4-manifolds. Using family Seiberg–Witten theory, we prove that for K3-type 4-manifolds and their blow-ups, certain order two subgroups cannot be realized even homotopy coherently. This in particular gives an alternative proof that the classical Nielsen realization fails for these manifolds \cites{BKK3,FarbK3,konno2024dehn,KMTrealSW}. In contrast, we also construct non-kinetic homotopy coherent actions that are not induced by genuine group actions. 
\end{abstract}

\section{Introduction}

Let $M$ be a closed oriented smooth manifold, and let $\Diff(M)$ denote its diffeomorphism group. 
The mapping class group of $M$ is defined as $\MCG(M):=\pi_0(\Diff(M))$. 
Given a finite subgroup $G\subset \MCG(M)$, the classical Nielsen realization problem asks whether G can be realized by a genuine smooth action on M:

\begin{Question}[Nielsen Realization Problem]
   Does the subgroup $G\hookrightarrow\MCG(M)$ admit a lift to $\Diff(M)$? Equivalently, does there exist a homomorphism $G\to\Diff(M)$ making the following diagram commute?
    \begin{equation*}
        \xymatrix{
            G \ar@{-->}[r] \ar@{=}[d] & \Diff(M) \ar[d] \\
            G \ar[r] & \pi_0(\Diff(M))
        }
    \end{equation*}
If such a lift exists, we say that $G$ is \textbf{(Nielsen) realizable}. 
\end{Question}

The smooth Nielsen realization problem has been studied extensively in various dimensions. When $M$ is a closed oriented surface, Kerckhoff \cite{dim2Nielsen} proved that every finite subgroup of $\MCG(M)$ is realizable. In contrast, in dimensions at least three,  Raymond--Scott \cite{highdimNielsen} showed that Nielsen realization fails for certain nilmanifolds.

The four-dimensional case has received considerable attention in recent years. Baraglia--Konno \cite{BKK3} and Farb--Looijenga \cite{FarbK3} independently established the failure of Nielsen realization for $K3$ surfaces. Subsequently, several authors \cite{konno2024dehn,KMTrealSW,arabadji2023spin,baraglia2026mapping} constructed many further examples of smooth $4$-manifolds for which Nielsen realization fails.


When a finite subgroup $G\subset \pi_0(\Diff(M))$ fails to be realized by a genuine action, it is natural to ask whether it can instead be realized by a homotopy coherent action.

The notion of a homotopy coherent action arises naturally in situations where genuine group actions are too rigid and algebraic identities are replaced by coherent systems of homotopies. Such actions have recently found a number of applications in topology and received considerable attention; see, for example,  \cite{HLS-HF,BSU2,KPTcoherent,KPTstablization,coherentS2,KPTkinetic}. These developments motivate the following homotopy coherent Nielsen realization problem.
\begin{Definition}[Homotopy Coherent Nielsen Realization]
    Let $\iota\colon G\hookrightarrow\MCG(M)= \pi_0(\Diff(M))$ be a finite subgroup of the mapping class group. We say that $G$ is \textbf{homotopy coherently Nielsen realizable} if there exists a map
    $H\colon BG \to B\Diff(M)$
such that the following diagram commutes:
    \[
\begin{tikzcd}[column sep=large, row sep=large]
\pi_1(BG) \arrow[r, "{H_*}"{above, pos=0.5}] \arrow[d, equal]
& \pi_1(B\Diff(M)) \arrow[d, equal] \\
G \arrow[r, "\iota"{above, pos=0.5}]
& \pi_0(\Diff(M))
\end{tikzcd}
\]
Such a map $H$ is called a \textbf{homotopy coherent realization} of $G$.
\end{Definition}

We mainly study the underlying smooth $4$-manifold of a $K3$ surface. In fact, our results apply more generally to  a class of smooth $4$-manifolds sharing the same basic topological invariants as $K3$ surfaces.
\begin{Definition}[$K3$-type manifold]
    A closed oriented smooth $4$-manifold $M$ is said to be of \textbf{$K3$-type}, if 
    \begin{enumerate}
        \item $b_1(M)=0, b_2^+(M) = 3, b_2^-(M) = 19$;
        \item $M$ admits a unique spin $\spinc$ structure.
    \end{enumerate}
\end{Definition}

Following \cite{FarbK3,konno2024dehn,KMTrealSW}, we consider two types of order $2$ mapping classes. One is represented by Dehn twist $\tau_S$ along an embedded $(-2)$-sphere $S$, and the other by reflection $r_E$ along an embedded $(-1)$-sphere $E$. We refer to Sections \ref{sec:Dehn} and \ref{sec:reflection}, respectively, for their precise definitions. The following is our main theorem. 
\begin{Theorem}
    \label{thm:main thm}
    Let $M$ be a $K3$-type $4$-manifold.
    \begin{enumerate}[label=(\arabic*)]
        \item  \label{thm:main-1}Let $S\subset M$ be a smoothly embedded $(-2)$-sphere. Then the order $2$ subgroup $G$ generated by the mapping class $[\tau_S]$ of the Dehn twist along $S$ is not homotopy coherently Nielsen realizable.
        \item \label{thm:main-2}Consider the blown-up manifold $M\# n\overline{\CP^2}$. Let $E=E_1\sqcup\cdots\sqcup E_n$ be the embedded $(-1)$-spheres. Then the order $2$ subgroup $G$ generated by the mapping class $[r_E]$ of the simultaneous reflection along $E_1,\cdots, E_n$ is not homotopy coherently Nielsen realizable.
    \end{enumerate}
\end{Theorem}

The main ingredient of the proofs is a gluing formula for family Seiberg--Witten invariants proved by J. Lin \cite{linFSW}.
Although our approach also relies on Seiberg--Witten theory, it provides an alternative proof of some of the results in \cite{konno2024dehn,KMTrealSW}. 
\begin{rmk}
    Though Konno--Miyazawa--Taniguchi only considered the reflection along a single $(-1)$-sphere \cite[Theorem 1.16]{KMTrealSW}, their argument also applies to simultaneous reflections along multiple $(-1)$-spheres.
    However, the real condition required in their argument fails for $K3\# 16\overline{\CP^2}$, while our argument still works.
\end{rmk}

Another interesting case is the neck Dehn twist $\varphi$ along a smoothly embedded $S^3$. In contrast to the examples above, $\varphi$ is in fact homotopy coherently Nielsen realizable.
\begin{Theorem}
    \label{thm:K3}
    Let $M$ be a closed smooth $4$-manifold containing an embedded $S^3$, and $\varphi$ be the neck Dehn twist along this $S^3$.
   If $[\varphi]$ generates an order $2$ subgroup in $\MCG(M)$, then this subgroup is homotopy coherently Nielsen realizable.
\end{Theorem}
Together with \cite[Theorem 7.1]{konno2024dehn}, Theorem \ref{thm:K3} provides examples of \textbf{non-kinetic} homotopy coherent action, namely, homotopy coherent actions $H \colon BG \to B\Diff(M)$ that are not induced by any genuine group action $G\to \Diff(M)$. 
Moreover, we have the following corollary:
\begin{Corollary}
    \label{cor:nonkinetic}
    Let $M$ be a simply connected spin $4$-manifold with $\sigma(M)\neq 0$. 
    If the neck Dehn twist $\varphi$ along an embedded $S^3$ on $M$ is not smoothly isotopic to the identity, then the homotopy coherent action constructed in Theorem $\ref{thm:K3}$ is non-kinetic.

    In particular, $K3\# K3$ and $K3 \# K3 \# S^2\times S^2$ admit non-kinetic homotopy coherent $\Z/2$ actions.
\end{Corollary}
\begin{rmk} 
A different family of closed $4$-manifolds admitting non-kinetic homotopy coherent actions was given by Kang--Park--Taniguchi \cite[Theorem 1.2]{KPTkinetic}.
\end{rmk}

The paper is organized as follows. 
In Section \ref{sec:FSW}, we recall the definition and basic properties of the family Seiberg--Witten invariants. 
In Section \ref{sec:criterion}, we will see how family Seiberg--Witten invariant can obstruct the homotopy coherent Nielsen realizability.
In Section \ref{sec:Dehn}, we prove Theorem \ref{thm:main thm}\ref{thm:main-1}, which is the Dehn twist along $(-2)$-sphere case. 
In Section \ref{sec:reflection}, we prove Theorem \ref{thm:main thm}\ref{thm:main-2}, which is the reflection along $(-1)$-sphere case.
In Section \ref{sec:K3}, we prove Theorem \ref{thm:K3} and give an example of non-kinetic homotopy coherent action.

\textbf{Acknowledgement.} We would like to thank Jianfeng Lin for many enlightening discussions and Hokuto Konno for helpful explanations of his work. We also thank Sungkyung Kang, Xingpei Liu, Xiayu Tan for helpful discussions, and Seraphina Eun Bi Lee for pointing out an incorrect citation in an earlier version of this paper. The second author would also like to thank Yaoping Xie for helpful conversations. 
\section{Family Seiberg--Witten theory revisited}
\label{sec:FSW}
\subsection{Family Seiberg--Witten invariants}
In this section, we briefly recall the construction and properties of the family Seiberg--Witten invariants.

Let $M$ be a connected closed oriented smooth $4$-manifold, and let $B$ be a closed smooth oriented manifold. 
Consider a smooth fiber bundle $M \to \tilde{M} \xrightarrow{p} B$. Fix a basepoint $b_0 \in B$, and identify $M $ with the fiber $M_{b_0}$ as a submanifold of $\tilde M$. 
This is called a family of $4$-manifolds over $B$, and we will denote it by $\tilde M/B$. 
\begin{Definition}
    \label{def:admissible family}
    The family $\tilde M/B$ is said to be \textbf{admissible}, if 
    \begin{itemize}
        \item the monodromy action $\pi_1(B,b_0)$ on $H_*(M;\Z)$ preserves the homology orientation, i.e., the orientation on $H^0(M;\R) \oplus H^1(M;\R) \oplus H^2_+(M;\R)$;
        \item either $b_2^+(M)>\dim B +1 $ or the monodromy action is trivial.
    \end{itemize}
\end{Definition}
In this paper, we will focus mainly on the case $B=S^2$. Note that in this case, the families are automatically admissible since $S^2$ is simply-connected.

Next, we define the Seiberg--Witten configuration space for families.
\begin{Definition} 
    \label{def:configuration on families}
    Let $\tilde M/B$ be a family over $B$.
    \begin{enumerate}
        \item $T(\tilde M/B)\coloneqq \ker(T\tilde M \xrightarrow{p_*} TB)$ is called the vertical tangent bundle of the family. 
        We will always equip the vertical tangent bundle with a Riemannian metric $g_{\tilde M/B} = \{g_b\}_{b\in B}$.
        \item A family $\spinc$ structure $\tilde{\fr{s}}$ is a lift of the orthonormal frame bundle $\Fr(\tilde M/B) \to \tilde M$ to a $\spinc(4)$-principal bundle $P$.
        \item Associating to $P$ via the standard action $\spinc(4) \curvearrowright \C^2_+ \oplus \C^2_-$, we obtain a spinor bundle $S \to \tilde M$ which canonically decomposes as $S^+ \oplus S^-$.
        \item The configuration space $\cC(\tilde M/B)$ is defined as $\bigcup\limits_{b\in B} \cC(M_b)$, where $\cC(M_b)$ is the classical Seiberg--Witten configuration space of $(M_b,g_b,\fr{s}_b)$.
    \end{enumerate}
\end{Definition}

It is well-known that for $(M,\fr{s})$, the virtual dimension for the moduli space of solutions to the Seiberg--Witten equation is  
\begin{equation}
    d(M,\fr{s}) \coloneqq \dfrac{c_1^2(\fr{s}) - 3\sigma(M) - 2\chi(M)}{4}.
\end{equation}
Similarly,  the virtual dimension for a family $(\tilde M, \tilde{\fr{s}})$ is
\begin{equation}
    d(\tilde M, \tilde{\fr{s}}) \coloneqq  d(M,\fr{s}) + \dim_{\R} B.
\end{equation}

Given a smooth perturbation family $\tilde \omega = \{\omega_b\}_{b\in B}$ of imaginary valued self-dual $2$-forms on the fibers $\{M_b\}$, i.e., $\omega_b \in \Omega^2_+(M_b;\R)$, we say a configuration $(b,A_b,\phi_b) \in \cC(\tilde M/B)$ satisfies the perturbed Seiberg--Witten equation if 
\begin{equation}
    \label{eq:SW equation}
    \begin{cases}
        F^+_{A^t_b} + \rho^{-1}(\phi_b\phi_b^*)_0 = \ii\omega_b, \\
        D_{A_b}^+\phi_b = 0,
    \end{cases}
\end{equation}
where $(\phi\phi^*)_0$ denotes the traceless part of the Hermitian endomorphism of $S^+$.

The following definition is the analogue of \cite{KMtextbook}*{Definition 22.1.1} for families.
\begin{Definition}
    \label{def:admissible condition}
    We say that the perturbation $(g_{\tilde{M}/B,}\tilde \omega)$ is an \textbf{admissible pair}, if 
    \begin{enumerate}
        \item all zeros of the equation are non-degenerate;
        \item the moduli space is regular at each point, i.e., the linearization of the equation is surjective \cite[Definition 14.5.6]{KMtextbook};
        \item there are no reducible solutions.
    \end{enumerate}
\end{Definition}

As in the classical Seiberg--Witten theory, admissible pairs are also generic in the family setting. 
By the Sard-Smale theorem and the Atiyah-Singer index theorem, for an admissible pair $(g_{\tilde M/B}, \tilde\omega)$, the moduli space is a compact smooth oriented manifold of dimension $d(\tilde M, \tilde{\fr{s}})$. The orientation is determined by the homology orientation of $M$ and the orientation of the base $B$.
When $d(\tilde M, \tilde{\fr{s}}) = 0$, we use $\# \cM(\tilde M, \tilde{\fr{s}}, g_{\tilde M/B}, \tilde \omega) \in \Z$ to denote the number of points of the moduli space counted with sign.

Next, we consider the wall-crossing. 
Let $\cH^+_b \subset \Omega^2(M_b)$ be the space of self-dual harmonic $2$-forms on $M_b$ with respect to the metric $g_b$. 
Then $\cH^+ = \bigcup\limits_{b\in B} \cH^+_b$ is a vector bundle over $B$ of rank $b_2^+(M)$.
\begin{Definition}
    \label{def:chamber}
    A \textbf{chamber} of $\tilde M/B$ is a homotopy class of sections $B \to S(\cH^+)\simeq \mathcal{H}^+\setminus 0_B$, where $S(\cH^+)$ denotes the unit sphere bundle and $0_B$ denotes the zero section.
    We use $\mathcal{CH}(\tilde M/B)$ to denote the set of chambers.

    In particular, when the monodromy is trivial and $b_2^+(M)\geq 2$, the bundle $S(\cH^+)$ is trivial with fiber $S^{b_2^+(M)-1}$. 
    In this case, a constant section determines a canonical chamber, denoted as $\xi_c$.
\end{Definition}

Given a pair $(g_{\tilde M/B},\tilde\omega)$, we assign to it a chamber $\xi(g_{\tilde M/B},\tilde \omega) \in \mathcal{CH}(\tilde M/B)$ by 
\begin{equation}
    b \mapsto \pr_{\cH^+_b} (\omega_b + 2\pi c_1(\fr{s})),
\end{equation}
where $\pr_{\cH^+_b}$ is the $L^2$-orthogonal projection from $\Omega^2(M_b)$ onto $\cH^+_b$.

The following proposition of Li--Liu shows that $\# \cM$ only depends on the chamber.
\begin{Proposition}
    \label{prop:FSW and chamber}\cite{LiLiu}
    Assume that $b_2^+(M)\geq 1$.
    \begin{enumerate}
        \item For every $\xi_0 \in \mathcal{CH}(\tilde M/B)$, there exists an admissible pair $(g_{\tilde M/B}, \tilde \omega)$ such that $\xi(g_{\tilde M/B},\tilde\omega) = \xi_0$.
        \item If $d(\tilde M, \tilde{\fr{s}})=0$, then $\#\cM(\tilde M, \tilde{\fr{s}}, g_{\tilde M/B}, \tilde \omega)$ depends only on $\xi(g_{\tilde M/B},\tilde \omega)$. 
    \end{enumerate}
\end{Proposition}

Now we can define the family Seiberg--Witten invariants.
\begin{Definition}
    \label{def:FSW}
    Assume that $b_2^+(M)\geq 1$. Suppose $d(\tilde M, \tilde{\fr{s}})=0$ and let $\xi \in \mathcal{CH}(\tilde M/B)$. The family Seiberg--Witten invariant is defined by 
    \begin{equation*}
        \FSW_\xi(\tilde M, \tilde{\fr{s}}) \coloneqq \#\cM(\tilde M, \tilde{\fr{s}}, g_{\tilde M/B}, \tilde \omega)
    \end{equation*}
    for any admissible pair $(g_{\tilde M/B},\tilde \omega)$ with $\xi(g_{\tilde M/B},\tilde \omega)= \xi$.
    
    When $b_2^+\geq 2$ and the monodromy is trivial, we simply write $\FSW(\tilde M, \tilde{\fr{s}})$ for $\FSW_{\xi_c}(\tilde M, \tilde{\fr{s}})$.
\end{Definition}

The family Seiberg--Witten invariant can detect the nontriviality of a family, as illustrated by the following proposition. 
\begin{Proposition}
    \label{prop:vanish for trivial bundle}
    Assume $b_2^+(M)>1$. Consider the trivial family $\tilde M = M \times B \to B$.
    Then $\FSW(\tilde M, \tilde s) =0$ for any trivial family $\spinc$ structure $\tilde{\fr{s}}$ satisfying $d(\tilde M, \tilde s) = 0$.
\end{Proposition}
\begin{proof}
    Let $\fr{s} = \tilde{\fr{s} }|_M$, then $d(M,\fr{s}) = -\dim_\mathbb{R}B$. 
    By classical Seiberg--Witten theory, there exists a pair $(g,\omega)$ on $M$ such that $\cM(M,\fr{s},g,\omega)$ is empty. 
    Define the family Riemannian metric $g_{M\times B/B}$ and perturbation $\tilde{\omega}$ simply by taking $g,\omega$ on each fiber, hence $\cM(M\times B, \tilde{\fr{s}}, g_{M\times B/B},\tilde \omega)$ is also empty. 
    Note that $(g_{M\times B/B},\tilde \omega)$ is admissible, and $\xi(g_{M\times B/B},\tilde \omega)$ represents the canonical chamber.
    Therefore, $\FSW(M\times B, \tilde{\fr{s}}) = 0$.
\end{proof}

\subsection{Family Seiberg--Witten invariants over $S^2$}
In this section, we focus on smooth families over $S^2$ whose fiber $M$ is a connected closed oriented smooth 4-manifold with $b_1(M)=0$ and $b_2^+(M)>1$.
It is known that smooth $M$-family over $S^2$ can be classified by $\pi_1(\Diff(M),\id)$. 
Therefore, once $\gamma$ is fixed, the family Seiberg--Witten invariant only depends on the choice of the family $\spinc$ structure.
The following lemma shows that its value actually depends only on the $\spinc$ structure on the fiber.

\begin{lemma}
    \label{lemma:FSW spinc only on fiber}
    Let $M$ be an oriented closed smooth 4-manifold with $b_1(M)=0$ and $b_2^+(M)>1$. Fix any smooth family $M \xrightarrow{\iota} \tilde M \xrightarrow{p} S^2$ and any $\spinc$ structure $\fr{s}_0$ on $M$.
    \begin{enumerate}
        \item There exists a family $\spinc$ structure $\tilde{\fr{s}}_0$ such that $\tilde{\fr{s}}_0|_M = \fr{s}_0$.
        \item For any two extensions $\tilde{\fr{s}}_1, \tilde{\fr{s}}_2$, there exists a complex line bundle $L \to S^2$ such that $\tilde{\fr{s}}_1 = \tilde{\fr{s}}_2 \otimes p^*L$.
        \item Suppose $d(M,\fr{s}_0)=-2$, then $\FSW(\tilde M,\tilde{\fr{s}}_1) = \FSW(\tilde M, \tilde{\fr{s}}_2)$.
    \end{enumerate}
\end{lemma}
\begin{proof}
    (1) and (2) in the lemma are special cases of \cite{baraglia2019obstructions}*{Proposition 2.1}. We prove (3). Since $d(\tilde M,\tilde{\fr{s}_i})=0$ for $i=1,2$, the moduli spaces $\cM(\tilde{\fr{s}}_i)$ are finite sets. 
    Therefore, only finitely many points $b_1,\dots,b_n$ on $S^2$ occur as the base points of elements of the moduli spaces. 
    For each $b_j$, choose a nonzero vector $v_j \in L_{b_j}$ which gives an identification $L_{b_j} \cong \C$. 
    Then the bundle map between the bundles over $M_{b_j}$ 
    \begin{align*}
        S_{b_j} & \to S_{b_j} \otimes L_{b_j} \\
        \phi & \mapsto \phi \otimes v_j
    \end{align*}
    induces an isomorphism $\tilde{\fr{s}}_1|_{M_{b_j}} \cong \tilde{\fr{s}}_2|_{M_{b_j}}$ for each $j$.
    
     Therefore, $\#\cM(M_{b_j},\tilde{\fr{s}}_1|_{M_{b_j}}) = \#\cM(M_{b_j},\tilde{\fr{s}}_2|_{M_{b_j}})$. We have
    \begin{align*}
        & \FSW(\tilde M, \tilde{\fr{s}}_1) = \# \cM(\tilde M, \tilde{\fr{s}}_1) = \sum\limits_{j=1}^n \#\cM(M_{b_j},\tilde{\fr{s}}_1|_{M_{b_j}}) \\
        = & \sum\limits_{j=1}^n \#\cM(M_{b_j},\tilde{\fr{s}}_2|_{M_{b_j}}) =  \# \cM(\tilde M, \tilde{\fr{s}}_2) =\FSW(\tilde M, \tilde{\fr{s}}_2).
    \end{align*}
    This proves (3).
\end{proof}
\begin{rmk}
    In fact, (3) holds for more general base spaces than $S^2$.
\end{rmk}

By the preceding lemma, the following invariant is well defined. 
\begin{Definition}
    \label{def:FSW for loop}
    For any $\spinc$ structure $\fr s$ on $M$ with $d(M,\fr{s})=-2$, and any $\gamma \in \pi_1(\Diff(M),\id)$, define 
\begin{equation*}
        \SW(\gamma,\fr{s}) \coloneqq \FSW(\tilde M, \tilde{\fr{s}}),
    \end{equation*}
    where $\tilde M\to S^2$ is the family determined by $\gamma$ and $\tilde{\mathfrak{s}}$ is any family $\spinc$ structure extending $\mathfrak{s}$. 
\end{Definition}

The following three properties of $\SW(\gamma, \fr{s})$ were established by Baraglia.
\begin{Proposition}\label{Prop: group homomorphism}
    \cite[Theorem 2.6]{BaragliaK3}
    If $b_2^+(M)\geq 2$, then for each $\spinc$ structure with $d(M,\fr{s})=-2$, the map 
    \begin{equation*}
        SW(-,\fr{s})\colon \pi_1(\Diff(M),\id) \to \Z 
    \end{equation*} 
    is a group homomorphism.
\end{Proposition}

\begin{Proposition}
    \label{lemma:conjugation on SW}
    \cite[Proposition 2.7]{BaragliaK3}
    Let $\fr{s}$ be a $\spinc$ structure on $M$ with $d(M,\fr{s})=-2$. For any $f\in \Diff(M)$ that preserves the homology orientation of $M$ and $\gamma \in \pi_1(\Diff(M))$, we have 
    \begin{equation}
        \SW(\gamma,\fr{s}) = \SW(f^{-1}\gamma f,f^*\fr{s}),
    \end{equation}
    where $(f^{-1}\gamma f)(t) = f^{-1}\circ \gamma(t) \circ f \in \Diff(M)$ for $t \in S^1$. 
\end{Proposition}

\begin{Proposition}\cite[Proposition 2.8]{BaragliaK3}\label{Prop: charge conjugation}
    Let $\fr{s}$ be a $\spinc$ structure on $M$ with $d(M,\fr{s})=-2$, $\gamma \in \pi_1(\Diff(M))$.  Then we have
    \[\SW(\gamma,\mathfrak{s})=(-1)^{\frac{b_2^+(M)-1}{2}}\SW(\gamma,\bar{\mathfrak{s}})\]
\end{Proposition}
In \cite{linFSW}, J. Lin proved a gluing theorem for family Seiberg--Witten invariants using family monopole Floer homology. The following proposition is a special case of that theorem tailored to our setting.
\begin{Proposition}\cite[Theorem M]{linFSW}\label{Prop: gluing}
    Let $M$ be a closed oriented smooth 4-manifold with $b_1(M)=0$ and $b_2^+(M)>1$. 
    Suppose $M$ admits a decomposition $M=M_1\bigcup\limits_YM_2$, where $Y$ is an L-space, $b_1(M_1)=0$,  $b_2^+(M_1)=0$ and $b_2^+(M_2)>1$. 
    Let $[\alpha]\in\pi_1(\Diff(M))$ be represented by a loop of diffeomorphisms supported in $\mathrm{Int}(M_1)$;
    Let $M\to \tilde{M}\to S^2$ be the smooth family associated to the clutching function $\alpha$.
    Then $\tilde M$ admits a splitting  $\tilde{M}=\tilde{M}_1\cup \tilde{M}_2$, where $M_1\to \tilde{M}_1\to S^2$ is the bundle with clutching function $\alpha$ and $\tilde{M}_2=M_2\times S^2$ is the trivial bundle;
    Let $\mathfrak{s}_0$ be a $\spinc$ structure on $M$ and $\tilde{\mathfrak{s}}$ be a family $\spinc$ structure on $\tilde{M}$ such that $d(M,\mathfrak{s}_0)=d(\tilde M, \tilde{\mathfrak{s}})=0$ and $\tilde{\mathfrak{s}}|_{\tilde{M}_2}$ is the pullback of $\mathfrak{s}_0|_{M_2}$.  
    Then we have
    \[\FSW(\tilde M,\tilde{\mathfrak{s}})=-\langle c_1(\ind_\mathbb{C}(\slashed{D}^+(\tilde{M},\tilde{\mathfrak{s}})), [S^2]\rangle\cdot\SW(M,\mathfrak{s}_0).\]
\end{Proposition}

\section{An obstruction to homotopy coherent realizability}
\label{sec:criterion}
Let $G$ be a finite group. Equip $BG$ with its standard CW structure arising from the bar construction, and denote its $n$-skeleton by $(BG)_n$. The extent to which $G$ admits a homotopy coherent realization can be measured by considering the skeleta of $BG$.
\begin{Definition}
Let $M$ be a smooth manifold, and let $\iota\colon G\hookrightarrow\pi_0(\Diff(M))$ be a finite subgroup of the mapping class group. For an integer $n\ge2$, we say that $G$ is \textbf{$\boldsymbol{n}$-realizable} if there exists a map $H_n \colon (BG)_n \to B\Diff(M)$ such that the diagram
\[
\begin{tikzcd}[column sep=large, row sep=large]
\pi_1((BG)_n) \arrow[r, "{(H_n)_*}"{above, pos=0.5}] \arrow[d, equal]
& \pi_1(B\Diff(M)) \arrow[d, equal] \\
G \arrow[r, "\iota"{above, pos=0.5}]
& \pi_0(\Diff(M))
\end{tikzcd}
\]
commutes. Such a map $H_n$ is called an \textbf{$\boldsymbol{n}$-realization} of $G$. Note that $G$ is always $2$-realizable by definition. 
\end{Definition}

We now derive an obstruction to the $3$-realizability of any order $2$ subgroup in $\pi_0(\Diff(M))$ using family Seiberg--Witten invariants. 

Let $[\tau]\in\MCG(M)$ be an element of order 2, represented by a diffeomorphism $\tau$. Then $\tau^2$ is smoothly isotopic to the identity. 
If $\langle[\tau]\rangle$ is homotopy coherently Nielsen realizable, then there exists a smooth $M$-bundle over $\mathbb{RP}^\infty$ whose restriction to the $1$-skeleton is isomorphic to the mapping torus of $\tau$
\[T\tau=M\times\mathbb{R}/(x,t+1)\sim (\tau(x),t).\]
Pulling back the mapping torus of $\tau$ along the double cover $S^1\to\mathbb{RP}^1$, we obtain the mapping torus of $\tau^2$
\[T\tau^2=M\times\mathbb{R}/(x,t+1)\sim (\tau^2(x),t).\]
Here the induced covering map sends $[(x,t)]$ to $[(x,2t)]$. An extension of $T\tau$ to a bundle over $\mathbb{RP}^2$ is determined by a choice of smooth isotopy from $\id$ to $\tau^2$. Pick any smooth isotopy 
$f:M\times[0,1]\to M$ with $f_0=\id_M, f_1=\tau^2$ and extend it to $t\in \mathbb{R}$ by setting $f_{t+1}(x)=f_t(\tau^2(x))$.  The isotopy $f_t$ induces a trivialization of $T\tau^2$
\begin{align*}
    \psi\colon T\tau^2 & \to M\times S^1 \\
    [(x,t)] & \mapsto [(f_t(x),t)].
\end{align*}
Therefore, we can fill in $M\times D^2$ to get a $M$-bundle over $\mathbb{RP}^2$, denoted by $E$. Denote the pull-back of $E$ along the double cover $S^2\to \mathbb{RP}^2$ by $\tilde{M}$. 

Next, we compute the clutching function for $\tilde{M}$. Write $S^2=D_+\cup D_-$ and identify $\partial D_- = \partial D_+ = S^1 = \R/\Z$.
The gluing map of the upper hemisphere is
\begin{align*}
    \psi^{-1}\colon M\times \partial D_+ & \to T\tau^2\\
    (x,t) & \mapsto(f_t^{-1}(x),t),
\end{align*}
while gluing map for the lower hemisphere is the conjugation of $\psi^{-1}$ using deck transformations:
\[M\times \partial D_-\to M\times\partial D_-\xrightarrow{\cong}  T\tau^2\to T\tau^2\]
\[(y,t)\mapsto(y,t+1/2)\mapsto(f^{-1}_{t+1/2}(y),t+1/2)\mapsto(\tau f_{t+1/2}^{-1}(y),t).\]
Therefore, the clutching function of $\tilde{M}$ is the transition function  between the above two trivializations, which is 
\begin{align*}
    \gamma:S^1 & \to \Diff(M) \\
    t & \mapsto f_t\tau f_{t+1/2}^{-1}.
\end{align*}
However, $\gamma$ is not based at the identity and $\gamma_0=\gamma_1=\tau f_{1/2}^{-1}$. We therefore define 
\begin{equation}
    \label{eq:clutching function}
    \tilde \gamma_t\coloneqq\gamma_t f_{1/2}\tau^{-1}=f_t\tau f_{t+1/2}^{-1}f_{1/2}\tau^{-1}
\end{equation}
 which is a loop of diffeomorphisms based at the identity.  Note that $\pi_1(\Diff(M),\id)$ is abelian since $\Diff(M)$ is a topological group, and loop concatenation is homotopic to pointwise multiplication. 
 
The following lemma describes the dependence of the clutching function $\tilde{\gamma}_t$ on the isotopy $f_t$.
\begin{Lemma}\label{lemma: depend on isotopy}
    Let $g_t=\xi_t f_t$ be another isotopy  from $\id$ to $\tau^2$, where $\xi_t$ is  a loop of diffeomorphisms based at identity, and denote $\eta_t=g_t\tau g_{t+1/2}^{-1}$. Then the corresponding based clutching function is given by:
    \[\tilde\eta_t=\eta_t\eta_0^{-1}=\xi_t\tilde\gamma_t\gamma_0\xi_{t+1/2}^{-1}\xi_{1/2}\gamma_0^{-1}.\]
    In particular, \[[\tilde\eta]=[\xi]+[\tilde\gamma]+[\gamma_0\xi^{-1}\gamma_0^{-1}]\in \pi_1(\Diff(M),\id).\]
\end{Lemma}
\begin{proof}
    A direct computation gives:
    \begin{align*}
        \tilde{\eta}_t&=g_t\tau g_{t+1/2}^{-1} g_{1/2}\tau^{-1}\\
        &=\xi_t f_t \tau f_{t+1/2}^{-1}\xi_{t+1/2}^{-1}\xi_{1/2} f_{1/2}\tau^{-1}\\
        &=\xi_t\gamma_t\xi_{t+1/2}^{-1}\xi_{1/2}\gamma_0^{-1}\\
        &=\xi_t\tilde{\gamma}_t\gamma_0\xi_{t+1/2}^{-1}\xi_{1/2}\gamma_0^{-1}.
    \end{align*}
    Note that $[\xi_{t+1/2}^{-1}\xi_{1/2}]=[\xi_t^{-1}]$ in $\pi_1(\Diff(M),\id)$, which completes the proof.
\end{proof}
The family Seiberg--Witten invariants of the loops in the above Lemma \ref{lemma: depend on isotopy} satisfy the following relation.
\begin{Corollary}\label{relation of SW}
Let $\tau$ be a diffeomorphism of $M$ such that $\tau^2$ is smoothly isotopic to identity and preserves the homology orientation of $M$. 
Let $\tilde \eta, \xi,\tilde{\gamma}$ be the loops of diffeomorphisms in Lemma \ref{lemma: depend on isotopy}. Let $\mathfrak{s}$ be a $\spinc$ structure on $M$ with $d(M,\mathfrak{s})=-2$. Then
$$\SW([\tilde{\eta}],\mathfrak{s})=\SW([\xi],\mathfrak{s})+\SW([\tilde{\gamma}],\mathfrak{s})-\SW([\xi],\tau^*\mathfrak{s}).$$
\end{Corollary}
\begin{proof}
    Since $\gamma_0=\tau f_{1/2}^{-1}$ is isotopic to $\tau$, $\gamma_0$ also preserves the homology orientation of $M$. So the corollary follows from Lemma \ref{lemma: depend on isotopy}, Proposition \ref{Prop: group homomorphism} and Proposition \ref{lemma:conjugation on SW}.
\end{proof}
Therefore, if one could show that $\SW([\tilde{\eta}],\mathfrak{s})$ is always nonzero for some suitable $\spinc$ structure $\mathfrak{s}$ no matter how we change $[\xi]$, then the $M$-bundle over $S^2$ corresponding to $\tilde\eta$ is nontrivial. 
This could lead to the failure of the $3$-realizability as follows.
\begin{Lemma}
    \label{lem:criterion}
    In the setting of Corollary \ref{relation of SW}, if $\SW([\tilde{\eta}],\mathfrak{s})\neq 0$ for any choice of $\xi$, then $\langle [\tau]\rangle \subset \MCG(M)$ is not $3$-realizable.  
\end{Lemma}
\begin{proof}
    If there exists a smooth bundle $\tilde E \to \RP^3$ such that 
    \begin{gather*}
        \xymatrix{
            \tilde M \ar[r]\ar[d] & i^*\tilde E \ar[r]\ar[d] & \tilde{E}\ar[d] \\
            S^2 \ar[r] & \RP^2 \ar[r]^{i} & \RP^3,
        }
    \end{gather*}
    then $\tilde M \to S^2$ must be trivial since $\pi_2(\RP^3)=0$. 
    However, the nonvanishing of the family Seiberg--Witten invariant implies that $\tilde M \to S^2$ can not be trivial. This leads to a contradiction.
\end{proof}

Such a nonvanishing result for family Seiberg--Witten invariant is difficult to achieve in general, because we have no control over the loop $\xi$. 
However, 
we may use the conjugation symmetry together with mod $2$ rigidity to obtain such nonvanishing results. 
\section{Dehn twist along $(-2)$-sphere}
\label{sec:Dehn}
\subsection{Local model}
\subsubsection{Dehn twist}
Let us recall the construction of Dehn twist along a $(-2)$-sphere.
Let $M$ be a closed $4$-manifold and $S\subset M$ be an embedded $(-2)$-sphere. 
The tubular neighborhood $\nu(S)$ can be identified with the cotangent bundle $T^*S^2$ via an orientation-reversing diffeomorphism
\begin{equation*}
    \Phi \colon \Cl(\nu(S)) \xrightarrow{\cong} DT^*S^2,
\end{equation*}
where $\Cl$ denotes the closure.
Using the standard embedding $S^2 \subset \R^3$ and the standard metric to identify $T^*S^2$ with $TS^2$, 
we write 
\begin{equation}
    T^*S^2 = \{(q,p) \in \R^3 \times \R^3 \mid |q|=1, p\perp q \}.
\end{equation}
With respect to this metric , let $\lc$ denote the unit disk bundle.
The normalized geodesic flow away from the zero section is 
\begin{align*}
    \varphi_t \colon \lc - S^2 & \to \lc - S^2 \\
    (q,p) & \mapsto (\cos (t) q + \dfrac{1}{r} \sin (t)p, -r\sin(t)q + \cos (t)p),
\end{align*}
where $r = |p|$.
Let $\rho \colon [0,\infty) \to \R$ be a decreasing cut-off function such that $\rho(r) = \pi$ for $r < \delta$ and $\rho(r) = 0$ for $r>1-\delta$ for some $0<\delta <1$.
Then the local model of the Dehn twist is defined by 
\begin{equation}
    \label{eq:Dehn twist}
    \tau(q,p) = \begin{cases}
        \varphi_{\rho(r)}(q,p), & p\neq 0, \\
        (-q,0), & p =0. 
    \end{cases}
\end{equation}
Since $\varphi_{\pi}(q,p) = (-q,-p)$, we can extend it smoothly over the zero section by the antipodal map. Thus the above Dehn twist is well-defined and smooth. 
It is easy to see that the isotopy class of $\tau$ is independent of the choices. 

Using the orientation-reversing identification $\Cl(\nu(S)) \cong \lc$, we define the Dehn twist $\tau_S \colon M \to M$ to be 
\begin{equation}
    \tau_S(x) = \begin{cases}
        x, & x\in M-\Cl(\nu(S)), \\
        \tau(x), & x \in \Cl(\nu(S)) \cong \lc.
    \end{cases}
\end{equation} 
Since $\tau(q,p) = \id$ for $|p|>1-\delta$, $\tau_S$ is well-defined and smooth. 

\subsubsection{Isotopy of squared Dehn twist}
In fact, the isotopy class $[\tau_S]$ is an element of order $2$ in $\MCG(M)$. 
We will write an explicit isotopy from $\tau_S^2$ to $\id$ using Seidel's construction in \cite{seidel}.
Working in the local model, we have 
\begin{equation}
    \tau^2(q,p) = 
    \begin{cases}
        \varphi_{2\rho}(q,p), & p\neq 0, \\
        (q,0), & p = 0,
    \end{cases}
\end{equation}
where $\rho$ is short for $\rho(r)$. 
Hence $\tau^2$ is identity for $r<\delta$ and $r>1-\delta$ and $\tau^2$ preserves $r$.

Endow $T^*S^2$ with the standard symplectic form $\omega = \sum_{i=1}^3 \dd q_i \wedge \dd p_i$.
The geodesic flow $\varphi_t$ is then the Hamiltonian flow generated by the function $(q,p) \mapsto |p|$ on $T^*S^2 - S^2$.
Let $\beta$ denote the standard symplectic form on $S^2$ defined by $\beta(v)(X,Y) = \det(v,X,Y)$ for $v \in S^2 \subset \R^3$ and $X,Y \in T_vS^2 \subset \R^3$. The standard $\mathrm{SO}(3)$-action on $T^*S^2$ is given by:
\[A\cdot(q,p)=(Aq, Ap),\ \forall A\in \mathrm{SO}(3).\]
We then obtain an $SO(3)$-invariant family of symplectic forms $\omega_s \coloneqq \omega + s\pi^*\beta$ on $T^*S^2$, where $s\in \mathbb{R}$ is sufficiently small, and $\pi: T^*S^2\to S^2$ is the projection.

Next, we calculate the moment map for the standard $SO(3)$-action on $T^*S^2$.
\begin{Lemma}
    The moment map $\mu^s \colon T^*S^2 \to \fr{so}(3)^*$ of the standard action $SO(3) \to \mathrm{Symp}(T^*S^2,\omega_s)$ is 
    \[\mu^s(q,p)=sq+q\times p,\]
    where $\fr{so}(3)$ is identified with $\R^3$ via $\hat\xi  \mapsto \xi$
    such that $\hat\xi(v) = \xi \times v$, and $\mathfrak{so}(3)^*$ is identified with $\mathbb{R}^3$ using the standard inner product.
\end{Lemma}
\begin{proof}
    For any $\xi\in\mathfrak{so}(3)\cong \mathbb{R}^3$, it generates a vector field $\xi_{T^*S^2}$ on $T^*S^2$ defined by:
    \[\xi_{T^*S^2}(q,p)\coloneqq\frac{\dd}{\dd t}\Big |_{t=0} \exp(t\xi)\cdot(q,p)\]
    
Elementary computation shows that $\xi_{T^*S^2}(q,p)=(\xi\times q, \xi\times p)$. We need to show that for any $\xi\in\mathfrak{so}(3)$,
    \[\dd\langle \mu^s, \xi\rangle=\iota_{\xi_{T^*S^2}}\omega_s\]
    Given any $(u,v)\in T_{(q,p)}T^*S^2$ (viewing them as vectors in $\R^3$), $q\cdot u=0$.
    
    The left hand side is:
    \begin{align*}
        \dd\langle\mu^s,\xi\rangle(u,v)&=\dd\langle sq+q\times p,\xi\rangle(u,v)\\
        &=\langle su+u\times p, \xi\rangle+\langle q\times v,\xi\rangle
    \end{align*}

    The right hand side is:
    \begin{align*}
        \iota_{\xi_{T^*S^2}}\omega_s(u,v)&=(\omega+s\pi^*\beta)((\xi\times q,\xi\times p),(u,v))\\
        &=(\xi\times q)\cdot v-(\xi\times p)\cdot u+sq\cdot ((\xi\times q)\times u)
    \end{align*}
    where 
    \[sq\cdot ((\xi\times q)\times u)=sq\cdot((\xi\cdot u)q-(q\cdot u)\xi)=s(\xi\cdot u)q\cdot q=s\xi\cdot u.\]
    Applying the scalar triple product identity yields the equality.
\end{proof}

Since $\mu^s$ is nowhere zero for $s\neq 0$, the Hamiltonian $|\mu^s|$ generates a circle action $\sigma^s_t$ on the whole $T^*S^2$. In fact, we can describe this action explicitly as follows.
\begin{Lemma}\label{lemma: flow is rotation}
    The flow $\sigma^s_t (s\neq 0)$ associated with the Hamiltonian $H=|\mu^s|$ preserves $\mu^s/|\mu^s|$ and acts on $(q,p)$ by rotating about the axis $\frac{\mu^s}{|\mu^s|}(q,p)$ by angle $t$.
\end{Lemma}
\begin{proof}
     Consider 
    \begin{align*}
        \xi \colon T^*S^2 & \to \fr{so}(3)^* \cong \fr{so}(3) \cong \R^3 \\
        (q,p) & \mapsto \dfrac{\mu^s(q,p)}{|\mu^s(q,p)|}.
    \end{align*}
    $\forall A\in\mathrm{SO}(3)$, we have $\mu^s(A\cdot (q,p))=A(\mu^s(q,p))$. 
    By taking differentials, we have
    \[\dd\mu^s(X_{|\mu^s|})=\dd\mu^s(\xi_{T^*S^2})=\xi\times\mu^s=0.\]
    Therefore, $\mu^s$ and $\xi={\mu^s}/{|\mu^s|}$ are invariant under $\sigma^s_t$. The Hamiltonian equation is 
    \[\begin{cases}
        \dot{q}=\xi \times q\\
        \dot{p}=\xi \times p
    \end{cases}\]
    whose solution is the rotation about the axis $\xi$.
\end{proof}

As $s\to 0$, $\sigma^s$ converges to $\varphi$ on compact subsets of $T^*S^2-S^2$. 
Therefore, after choosing a small constant $\epsilon>0$, 
\begin{equation}
    \label{eq:isotopy}
    h_t(q,p) = \begin{cases}
        \sigma^{\epsilon}_{4t\rho} (q,p), & 0\leq t\leq 1/2, \\
        \sigma^{2(1-t)\epsilon}_{2\rho} (q,p), & 1/2\leq t\leq 1
    \end{cases}
\end{equation}
provides a smooth isotopy from $\id$ to $\tau^2$. 
Meanwhile, it is easy to check that $h_t|_{\partial \lc}$ is identity for every $t$, therefore $h_t$ extends by the identity to an isotopy from $\id$ to $\tau_S^2$ on the whole manifold $M$. We continue to denote this extended isotopy by $h_t$ and refer to it as Seidel's isotopy. 

\begin{rmk}
    To construct the isotopy, the most naive idea would be to replace the phases in the formula for $\tau^2$ by $2(1-t)\rho$. 
    But this construction fails since the resulting map does not extend smoothly across the zero section when $0<t<1$.
\end{rmk}

\subsection{Two special loops of diffeomorphisms}
\label{sec:two loops}
In this subsection, we will compare two families over $S^2$. 
The first one is obtained from Seidel's isotopy. The other is the $A_1$-family in \cite{linFSW}. 

\subsubsection{Seidel's clutching function and generalized Dehn twist}
Recall that in Section \ref{sec:criterion}, we have discussed how to obtain a family over $S^2$ from an isotopy between identity and $\tau^2$. Now let us consider the family over $S^2$ obtained from Seidel's isotopy, which is denoted by $h_t$. Plugging the formulas above into the expression of the clutching function $\gamma_t=h_t\tau h_{t+1/2}^{-1}$, we get:
\[\gamma_t=\begin{cases}
    \sigma^\epsilon_{4t\rho}\sigma_{\rho}\sigma^{(1-2t)\epsilon}_{-2\rho} ,&0\leq t\leq 1/2,\\
    \sigma^{2(1-t)\epsilon}_{2\rho}\sigma_{-\rho}\sigma^\epsilon_{(2-4t)\rho}, &1/2\leq t\leq 1.
\end{cases}\] 
Here $\rho$ stands for the radial cutoff function $\rho(r)$. Note that $\gamma_0=\gamma_1=\sigma_\rho \sigma^\epsilon_{-2\rho}$, we get a loop based at $\id$ by composing $\gamma_0^{-1}$:
\[\tilde{\gamma}_t=\begin{cases}
    \sigma^\epsilon_{4t\rho}\sigma_{\rho}\sigma^{(1-2t)\epsilon}_{-2\rho}\sigma^\epsilon_{2\rho}\sigma_{-\rho}, &0\leq t\leq 1/2,\\
    \sigma^{2(1-t)\epsilon}_{2\rho}\sigma_{-\rho}\sigma^\epsilon_{(4-4t)\rho}\sigma_{-\rho}, &1/2\leq t\leq 1.
\end{cases}\]
We call $\tilde{\gamma}$ Seidel's clutching function. Note that $\tilde\gamma$ preserves $|p|$.   
\begin{Remark}
    The preservation of $|p|$ is an essential condition, since in dimension $3$ a pseudo-isotopy which restricts to identity at both ends may not be straightened into an isotopy. 
\end{Remark}

Next we consider the $A_1$-family.
In \cite{linFSW},  J. Lin studied a similar loop of diffeomorphisms called the generalized Dehn twist, which arises naturally as the clutching function for the simultaneous resolution of the $A_1$-family over $S^2$. 

\begin{Definition}[Generalized Dehn twist]
    Let $S$ be a $(-2)$-sphere in $M$. As in Section \ref{sec:Dehn}, fix an identification $\Cl(\nu(S))\cong DT^*S^2$. Let $R_v(\theta)$ be the rotation in $\mathbb{R}^3$ around axis $v$ by angle $\theta$.
    The \textbf{generalized Dehn twist} along $S$ is a loop $\gamma_S\colon S^1 \to \Diff^+(M)$ constructed below.
    \begin{enumerate}
        \item For $0\leq |p|\leq 1/2$, we define 
    \[\gamma_S(t)(q,p)\coloneqq(q, R_q(4\pi t)(p)), \ t\in [0,1].\]

        \item Let $\Gamma_t$ be the restriction of $\gamma_S(t)$ to $\{(q,p) \in \lc \mid  |p|=1/2\}\cong\mathbb{RP}^3$. Note that $\Gamma$ is null-homotopic in $\pi_1(\Diff^+(\mathbb{RP}^3))\cong \mathbb{Z}/2\times\mathbb{Z}/2$. 
    Pick any based null-homotopy $\Gamma^s_t\in\Diff^+(\mathbb{RP}^3), (s,t)\in [0,1]^2$ such that $\Gamma^0_t=\Gamma_t$, $\Gamma^1_t= \id$, $\Gamma^s_0=\Gamma^s_1=\id$. For $1/2\leq |p|\leq 1$, we define
    \[\gamma_S(t)(q,p)\coloneqq\Gamma^{2|p|-1}_t(q,p), \ t\in [0,1].\]
    \end{enumerate}
    By choosing $\Gamma^s_t$ carefully, we may assume that $\gamma_S$ is a smooth loop of diffeomorphisms supported in $\nu (S)$, then we extend it by identity to a loop of diffeomorphisms on $M$. 
\end{Definition}
\begin{rmk}
    In fact, the homotopy class $[\gamma_S]\in \pi_1(\Diff^+(M))$ does not depend on the auxiliary choices in the construction above, which will be shown in the next part.
\end{rmk}

\subsubsection{Classification of special loops}
In this part, we will classify the loops in $\pi_1(\Diff^+_{\partial}(\lc))$ preserving $|p|$. 
Recall that the unit cotangent bundle $ST^*S^2\cong \mathbb{RP}^3\cong \mathrm{SO}(3)$. By generalized Smale conjecture for $\mathbb{RP}^3$ \cite{bamler2019ricci,ketover2025smale}, we have
$$\Diff^+(\mathbb{RP}^3)\simeq \mathrm{Isom}^+(\mathbb{RP}^3)\cong\mathrm{SO}(3)\times\mathrm{SO}(3),$$
and $(A,B)\in \mathrm{SO}(3)\times\mathrm{SO}(3)$ acts on $\mathbb{RP}^3\cong \mathrm{SO}(3)$ by $C\mapsto ACB^{-1}$. 

We unify the notation as follows:
\begin{itemize}
    \item The zero section of $DT^*S^2$ is denoted by $Z$.
    \item $G\subset \Diff^+_\partial(DT^*S^2)$ denotes the group of orientation-preserving diffeomorphisms of $DT^*S^2$ which preserve $|p|$, restrict to the identity near the boundary and fix $Z$ pointwise. Note that $\tilde{\gamma}$ and $ \gamma_S$ both represent elements in $\pi_1(G)$.  
    \item $H\subset G$ denotes the subgroup of diffeomorphisms in $G$ which are identities near the zero section. 
\end{itemize}

We have a natural decomposition 
\[TT^*S^2|_Z\cong TZ\oplus NZ\cong TS^2\oplus T^*S^2\]
Denote the projection onto the two summands by $\pi^{\parallel}$ and $\pi^{\perp}$ respectively. 
Given any diffeomorphism $f\in G$ and $(q,0)\in Z$,  write 
\[\dd f|_{(q,0)}=\begin{pmatrix}
    \id_{T_q Z} & *\\
    0 & \dd^\perp f|_{(q,0)}
\end{pmatrix}\]
with respect to the above decomposition, where $\dd^\perp$ means the $NZ$ component of the differential.  

\begin{Definition}[Winding number]
    Given any loop of diffeomorphisms $[\alpha]\in \pi_1(G,\id)$,  consider the composition:
    \[S^1\to \Diff^+_\partial(DT^*S^2)\to \Gamma(S^2, \Aut(TT^*S^2\Big|_{Z}) )\to \Gamma(S^2,\Aut(T^*S^2)) \]
    \[t\mapsto \alpha_t\mapsto \dd\alpha_t\Big|_{Z}\mapsto \dd^{\perp}\alpha_t\Big|_{Z}.\]
    Note that $\Gamma(S^2,\Aut(T^*S^2)) \simeq C^\infty(S^2, S^1)\simeq S^1$, the above map defines an element in $\pi_1(S^1) \cong \Z$.
    This number is called the \textbf{winding number} of $\alpha$, denoted by $w(\alpha)$.
    
    Since $w(\alpha)$ depends only on the homotopy class $[\alpha]\in \pi_1(G)$, we obtain a group homomorphism:
    \begin{align*}
        w: \pi_1(G) & \to \mathbb{Z},\\
        [\alpha] & \mapsto w([\alpha]).
    \end{align*}
\end{Definition}

The following observations together imply that loops in $G$ can be homotoped to standard forms.
\begin{Lemma}\label{lemma: trivial criterion}
    $\pi_1( H,\id)$ is trivial.
\end{Lemma}
\begin{proof}
    For any loop $[\alpha]\in\pi_1(H,\id)$ and any fixed $t\in [0,1]$, since $\alpha_t$ preserves $|p|$ and is identity near $Z$ and the boundary, we can view $\alpha_t$ as a loop of diffeomorphisms of $ST^*S^2\cong\mathbb{RP}^3$ based at $\id$.  Moreover,  $\alpha_0=\alpha_1=\id$, so the loop $\alpha$ can be viewed as a map $$S^2\to \Diff^+(\mathbb{RP}^3)\simeq\mathrm{SO}(3)\times \mathrm{
    SO}(3)$$
    Since $\pi_2(\mathrm{SO}(3))=0$, this map is nullhomotopic. Therefore,  $\alpha$ is also nullhomotopic in $H$.
\end{proof}

\begin{Lemma}\label{lemma: trivial differential}
    For any loop of diffeomorphisms $[\alpha]\in\pi_1(G,\id)$, if the normal derivatives  $\dd^\perp\alpha_t|_Z$ are identities at any point and any time, then $\alpha_t$ can be homotoped to a loop in $H$, which is an identity near $Z$.
\end{Lemma}
\begin{proof}
    The following construction is actually very general, but we only consider the case of $Z\subset T^*S^2$ here. For any $\phi\in G$ whose normal differential is identity on $Z$, write \

\[\phi(q,p)=(\phi_1(q,p), \phi_2(q,p)).\]
Define 
\[\phi^s(q,p)=\begin{cases}
    (\phi_1(q,sp), \frac{1}{s}\phi_2(q,sp)), & 0< s\leq 1,\\
    \id_{T^*S^2}, & s=0.
\end{cases}\]
Since $\dd^\perp\phi_t|_Z$ are identities, $\phi^s$ is smooth, preserves $|p|$ and $\phi^0=\id_{T^*S^2}$,  $\phi^1=\phi$. Let $X^s$ be the time-dependent vector field generating $\phi^s$. Consider a smooth cutoff function $\chi: DT^*S^2\to [0,1]$ such that for some small $0<\kappa<< 1$
\[\chi(q,p)=\begin{cases}
    1, &|p|\leq \kappa,\\
    0, & |p|\geq 2\kappa.
\end{cases}\]
Let $Y^s=\chi X^s$, and let $\psi^s$ be the flow generated by $Y^s$. Then $\psi^s$ is supported in $|p|\leq 2\kappa$,  preserves $|p|$ and is identity for $s=0$ while equal to $\phi$ for $s=1, |p|\leq \kappa$ . Then $\phi\circ (\psi^s)^{-1}$ is an isotopy in $G$ from $\phi$ to $\phi\circ(\psi^1)^{-1}$. Note that $\phi\circ (\psi^1)^{-1}\in H$. An $S^1$-parametric version of this construction yields the conclusion.
\end{proof}

\begin{Lemma}\label{lemma: standard form}
    Any class $[\alpha]\in \pi_1(G)$ with $w([\alpha])=n\in\mathbb{Z}$ can be represented by a loop $\alpha: S^1\to G$ based at $\id$ such that $\alpha$ restricts to the standard form
    \[\alpha_t(q,p)=(q,R_q(2n\pi t)(p) )\]
    near the zero section $Z$.
\end{Lemma}
\begin{proof}
    Let $\alpha$ be an arbitrary loop in $G$ based at identity with $w(\alpha)=n$. Let $A_t(q)\coloneqq\dd^\perp\alpha_t\Big|_{(q,0)}$ . Then $A$ is a loop in $\Gamma(S^2, \Aut(T^*S^2))\simeq S^1$ with winding number $n$. 
    
    Let $B_t(q)=R_q(2n\pi t)$ be the standard loop with winding number $n$ and let $C_t=B_t A_t^{-1}$. Since $C_t$ has the winding number $0$, we can write $C_t(q)=R_q(\theta(t,q))$ for some smooth function $\theta:[0,1]\times S^2\to \mathbb{R}$ with $\theta(0,q)=\theta(1,q)=0$. 
    
    Choose a smooth cutoff function $\lambda(|p|)$ which is $1$ for $|p|\leq \kappa$ and $0$ for $|p|\geq 2\kappa$. Define 
    \[\beta_t^s(q,p)\coloneqq(q, R_q(s\lambda(|p|)\theta(t,q))(p)),\]
    then $\beta^s_t\in G$, $\beta^0_t\equiv \id$, $\beta^s_0=\beta^s_1=\id$ and $\dd^\perp\beta^1_t|_Z=C_t$. Define $\alpha'_t\coloneqq\beta^1_t\circ\alpha_t$. Then $[\alpha']=[\alpha]\in \pi_1(G)$ and $\dd^\perp\alpha'_t|_Z=B_t$. 
    Let $L$ be the standard loop on $U= \{(q,p)\in DT^*S^2 \mid |p|< \kappa\}$ with
    \[L_t(q,p)=(q,R_q(2n\pi t)(p) ).\]
     Consider the loop $\delta_t\coloneqq L_t^{-1}\circ\alpha'_t$ on $U$, then $\dd^\perp\delta_t|_Z=\id$. Again by the rescaling and cut-off argument in Lemma \ref{lemma: trivial differential}, we get a $S^1$-parameter flow $\varphi^s_t\in G$ supported on $U$ such that $\varphi^0_t=\id$, $\varphi^1_t|_{U'}=\delta_t|_{U'}$ on the smaller neighborhood $U'=\{(q,p) \in U \mid |p|<\kappa/2\}$. Finally, let $\alpha''_t\coloneqq\alpha_t'\circ(\varphi^1_t)^{-1}$. Then $[\alpha'']=[\alpha]\in \pi_1(G)$ and $\alpha''_t|_{U'}=L_t|_{U'}$.
\end{proof}
\begin{Proposition}
    Half of the winding number gives an isomorphism
    \[w/2: \pi_1(G,\id)\cong \mathbb{Z},\  [\alpha]\mapsto w([\alpha])/2\]
    \end{Proposition}
\begin{proof}
    Since $(q,p)\mapsto (q, R_q(2\pi t)(p))$ for fixed $|p|>0$ is a nontrivial loop in $\Diff^+(\mathbb{RP}^3)$, by Lemma \ref{lemma: standard form},  any loop $\alpha$ in $G$ must have even winding number. Moreover, if $\alpha$ has zero winding number, then by Lemma \ref{lemma: standard form}  and Lemma \ref{lemma: trivial criterion}, $\alpha$ is null-homotopic in $G$. So $w/2$ is an injective homomorphism.  Since $w([\tilde{\gamma}])=2$, it's also surjective.
\end{proof}

\subsubsection{Winding number of the two special loops}
Now let us compute the winding number of our clutching function $\tilde{\gamma}$ and compare it to the generalized Dehn twist. 

Near the zero section, $\rho(|p|)=\pi$. So the formula of $\tilde{\gamma}$ simplifies to
\[\tilde{\gamma}_t=\begin{cases}
    \sigma^\epsilon_{4\pi t}, &0\leq t\leq  1/2,\\
    \sigma_{\pi}\sigma^\epsilon_{-4\pi t}\sigma_\pi, &1/2\leq t\leq 1
\end{cases}\]
for $|p|<\delta$.
\begin{Lemma}{\ }
\begin{enumerate}
    \item $\dd^\perp\sigma^\epsilon_t\Big|_{(q,0)}=R_{q}(t)$,\\
    \item     $\dd^\perp\tilde{\gamma}_t\Big|_{(q,0)}=
    R_{q}(4\pi t), \ 0\leq t\leq 1$.
\end{enumerate}
In particular, the winding number of $\tilde{\gamma}$ is equal to 2.
\end{Lemma}
\begin{proof}
    By Lemma \ref{lemma: flow is rotation}, we have
    \[\sigma^\epsilon_t(q,p)=(R_{\frac{\mu^\epsilon}{|\mu^\epsilon|}(q,p)}(t)(q),R_{\frac{\mu^\epsilon}{|\mu^\epsilon|}(q,p)}(t)(p)).\]
    For any $v\in NZ|_{(q,0)}\subset TT^*S^2|_{(q,0)}$, 
    \begin{align*}
        \dd^\perp\sigma^\epsilon_t\Big|_{(q,0)}(v) &=\pi^\perp(\frac{\dd}{\dd s}\Big|_{s=0}\sigma^\epsilon_t(q, sv)).\\
    &=\frac{\dd}{\dd s}\Big |_{s=0}R_{\frac{\mu^\epsilon}{|\mu^\epsilon|}(q,sv)}(t)(sv)\\
    &=R_{\frac{\mu^\epsilon}{|\mu^\epsilon|}(q,0)}(t)(v)\\
    &=R_{q}(t)(v)
    \end{align*}
    This completes (1).

    When $0\leq t \leq 1/2$, (1) immediately implies (2). 
    When $1/2\leq t \leq 1$, $\tilde{\gamma}_t=\sigma_\pi \sigma^\epsilon_{-4\pi t}\sigma_\pi$. Therefore, 
    \begin{align*}
        \dd^\perp\tilde{\gamma}_t\Big |_{(q,0)}(v)&= \pi^\perp(\dd\tilde{\gamma}_t\Big |_{(q,0)}(v))\\
        &=\pi^\perp(\dd\sigma_\pi\Big|_{(-q,0)}\circ \dd\sigma^\epsilon_{-4\pi t}\Big |_{(-q,0)}\circ \dd\sigma_\pi \Big|_{(q,0)}(v))\\
        &=\dd^\perp\sigma^\epsilon_{-4\pi t}\Big |_{(-q,0)}(v))\\
        &= R_{-q}(-4\pi t)(v)\\
        &=R_q(4\pi t)(v).
    \end{align*}
    This completes (2). And hence the winding number of $\tilde{\gamma}$ is equal to $2$ by definition.
\end{proof}

On the other hand, it is easy to see that the winding number of $\gamma_S$ is also $2$.
Therefore, we can conclude that: 
\begin{Corollary}
Seidel's clutching function $\tilde{\gamma}$ is identified with J.Lin's generalized Dehn twist $\gamma_S$ via homotopy near the $(-2)$-sphere $S$.  
\end{Corollary}
\subsection{Proof of Theorem \ref{thm:main thm}\ref{thm:main-1}}
\label{sec:proof}
Let $M\to \tilde{M}\to S^2$ be the smooth fiber bundle associated to the clutching function $\tilde{\gamma}$. Since $\tilde{\gamma}$ is supported in a tubular neighborhood $\Cl(\nu (S))\cong DT^*S^2$ of the $(-2)$-sphere $S$, the bundle $\tilde{M}$ has a natural splitting $\tilde{M}=\tilde{M}_1\cup \tilde{M}_2$, where $\nu (S)\to \tilde{M}_1\to S^2$ has clutching function $\tilde{\gamma}$ and $\tilde{M}_2=(M-\nu (S))\times S^2$ is a trivial bundle. Since $\tilde{\gamma}$ acts on the zero section $Z$ trivially, we obtain a trivial subbundle $Z \to \tilde S\cong S^2\times S^2 \to S^2$ inside $\tilde M_1$. Fixing any point $p\in Z$ gives us a section of $\tilde{S}$, denoted by $\Sigma_p$. Then 
\begin{gather*}
    Z\cdot Z = 0, Z\cdot \Sigma_p = 1, \Sigma_p\cdot \Sigma_p=0.
\end{gather*}
The following two lemmas are special cases of \cite{linFSW}*{Lemma 6.5, Proposition 6.6}. We rewrite the proof here for completeness.
\begin{Lemma}\label{lemma: intersection number} 
    \cite[Lemma 6.5]{linFSW}
    \begin{enumerate}
        \item $\tilde{S}\cdot\tilde{S}\cdot\tilde{S}= 8$
        \item $\langle p_1(T(\tilde M_1/S^2)),[\tilde S]\rangle= 8$
    \end{enumerate}
\end{Lemma}
\begin{proof}
    Let $N$ be the normal bundle of $\tilde{S}$ in $\tilde{M}_1$. Then $e(N|_Z)=-2$ and $e(N|_{\Sigma_p})=-w(\gamma_S)=-2$. So $e(N)=\PD(-2[Z]-2[\Sigma_p])$. And 
    \[\tilde{S}\cdot\tilde{S}\cdot\tilde{S}=\langle e(N)^2,[\tilde{S}]\rangle=8.\]
    This is the first statement. 
    
    To prove the second statement, note that there is a decomposition
    \[T(\tilde{M}_1/S^2)|_{\tilde{S}}=T(\tilde{S}/S^2)\oplus N.\]
    So \[p_1(T(\tilde{M}_1/S^2)|_{\tilde{S}})=p_1(T(\tilde{S}/S^2))+p_1(N)=e(T(\tilde{S}/S^2))^2+e(N)^2.\]
    Thus 
    \[\langle p_1(T(\tilde M_1/S^2)),[\tilde S]\rangle=\langle e(T(\tilde{S}/S^2))^2, [\tilde{S}]\rangle+\langle e(N)^2, [\tilde{S}]\rangle=8+\langle e(T(\tilde{S}/S^2))^2, [\tilde{S}]\rangle\]
    Since $e(T(\tilde{S}/S^2)|_Z)=2$ and $e(T(\tilde{S}/S^2)|_{\Sigma_p})=0$,  we have $e(T(\tilde{S}/S^2))=2\PD([\Sigma_p])$. So 
    \[\langle e(T(\tilde{S}/S^2))^2, [\tilde{S}]\rangle=0.\]
    Putting this into the formula above and using (1), we immediately conclude the statement.
\end{proof}

\begin{Lemma}\label{lemma: index}
    Let $\tilde{\mathfrak{s}}$ be a family $\spinc$ structure on $\tilde{M}$ such that $\PD(c_1(\tilde{\mathfrak{s}}|_{\tilde{M}_1}))= \pm 2[\tilde{S}]\in H_4(\tilde{M}_1,\partial\tilde{M}_1;\mathbb{Z})$, then 
    \[\langle c_1(\ind_\mathbb{C}(\slashed{D}^+(\tilde{M},\tilde{\mathfrak{s}})), [S^2]\rangle=\pm 1 .\]
\end{Lemma}
\begin{proof}
    By the family Atiyah-Singer index theorem, 
    \begin{align*}
        \langle c_1(\ind_\mathbb{C}(\slashed{D}^+(\tilde{M},\tilde{\mathfrak{s}}))),[S^2] \rangle&=\langle e^{c_1(\tilde{s})/2}\hat{A}(T(\tilde{M}/S^2)),[\tilde{M}]\rangle\\
        &=\frac{1}{48}\langle c_1(\tilde{\mathfrak{s}})^3-c_1(\tilde{\mathfrak{s}}) p_1(T(\tilde{M}/S^2)), [\tilde{M}]\rangle.
    \end{align*}
    Since $\PD(c_1(\tilde{\mathfrak{s}}|_{\tilde{M}_1}))=\pm 2[\tilde{S}]$, by Mayer-Vietoris sequence we have 
    \[c_1(\tilde{\mathfrak{s}})=\PD(\pm 2[\tilde{S}]+\beta\times[S^2])\]
    for some $\beta\in \mathrm{Im}(H_2(M-\nu (S);\mathbb{Z})\to H_2(M;\mathbb{Z}))$. Similarly to the computation in Lemma \ref{lemma: intersection number}, we have:
    \[[\tilde{S}]^i\cdot(\beta\times [S^2])^{3-i}=0,\  \forall 0\leq i<3\]
    and 
    \[\langle p_1(T(\tilde{M}/S^2)),\beta\times [S^2]\rangle=0.\]
    Therefore, 
    \[\langle c_1(\ind_\mathbb{C}(\slashed{D}^+(\tilde{M},\tilde{\mathfrak{s}}))),[S^2]\rangle=\pm \frac{1}{48}(8 [\tilde{S}]^3-2\langle p_1(T(\tilde{M}/S^2)), [\tilde{S}]\rangle)=\pm 1.\]
\end{proof}

\begin{Proposition}\label{Prop: nonvanish}
    Let $M$ be a K3-type 4-manifold containing a $(-2)$-sphere $S$. Let $[\tilde{\gamma}]=[\gamma_S]\in \pi_1(\Diff(M))$ be the generalized Dehn twist along $S$. Let $\mathfrak{s}_0$ be the self-conjugate $\spinc$ structure on $M$ and  $\mathfrak{s}=\mathfrak{s}_0\pm \PD([S])$. Then 
    $\SW([\tilde{\gamma}],\mathfrak{s})=\pm\SW(M, \mathfrak{s}_0)$ is odd.
\end{Proposition}
\begin{proof}
Pick any family $\spinc$ structure $\tilde{\mathfrak{s}}$ on $\tilde{M}$ such that $\tilde{\mathfrak{s}}|_M\cong \mathfrak{s}$.  By twisting $\tilde{\mathfrak{s}}$ with a complex line bundle pulled back from $S^2$, we may assume $\tilde{\mathfrak{s}}|_{\tilde{M}_2}$ is the pull-back of $\mathfrak{s}|_{M-\nu (S)}$. A direct computation shows that $d(M, \mathfrak{s}_0)=0$ and $d(\tilde{M},\tilde{\mathfrak{s}})=d(M,\mathfrak{s})+2=0$.  Thus, by Proposition \ref{Prop: gluing} and Lemma \ref{lemma: index} we have 
    \[\SW([\tilde{\gamma}],\mathfrak{s})=-\langle c_1(\ind_\mathbb{C}(\slashed{D}^+(\tilde{M},\tilde{\mathfrak{s}})), [S^2]\rangle \cdot \SW(M,\mathfrak{s}_0)=\pm \SW(M,\mathfrak{s}_0)\]
  By  \cite{morgan1997homotopy}*{Theorem 1.1, Remark 2.2},  $\SW(M,\mathfrak{s}_0)$ is odd, which completes the proof.
\end{proof}

We can then prove Theorem \ref{thm:main thm}\ref{thm:main-1}.
Let $\mathfrak{s}_0$ be the spin $\spinc$ structure on $M$. Let $\mathfrak{s}=\mathfrak{s}_0+\PD([S])$. Then $\tau^*\mathfrak{s}=\mathfrak{s}_0-\PD([S])=\bar{\mathfrak{s}}$, so we have 
\[\SW([\xi],\mathfrak{s})=-\SW([\xi],\tau^*\mathfrak{s})\]
by Proposition \ref{Prop: charge conjugation}.  Therefore, by Corollary \ref{relation of SW} and Proposition \ref{Prop: nonvanish},
\[\SW([\tilde{\eta}],\mathfrak{s})\equiv\SW([\tilde{\gamma}],\mathfrak{s})\neq 0 \mod 2\]
So by Proposition \ref{prop:vanish for trivial bundle}, regardless of the choice of isotopy, the corresponding family $\tilde M\to S^2$ is nontrivial. 
The theorem then follows by Lemma \ref{lem:criterion}.

\section{Reflection along $(-1)$-sphere}
\label{sec:reflection}
\subsection{Local model}
Let $M$ be a closed $4$-manifold and $E\subset M$ be an embedded $(-1)$-sphere. 
The tubular neighborhood $\nu(E)$ can be identified with the complex line bundle over $S^2$ of Euler number $(-1)$. 
There is an orientation-preserving diffeomorphism 
\begin{align*}
    \Phi\colon \Cl(\nu(E)) \xrightarrow{\cong} \overline{\CP^2}-D^4.
\end{align*}
Next we give an explicit model for $\overline{\CP^2}-D^4$. Fixing an $\varepsilon>0$, we have
\begin{gather*}
    U \coloneq \{[0:z_1:z_2] \} \cup \{[1:w_1:w_2] \mid r \geq \varepsilon \} \cong \overline{\CP^2}-D^4 \subset \overline{\CP^2},
\end{gather*}
where $r = |w_1|^2 + |w_2|^2$.
Also the tubular neighborhood of $\partial U$ can be written as 
\begin{gather*}
    \{[1:w_1:w_2] \mid \varepsilon \leq r \leq2\varepsilon\} \cong S^3 \times [0,1].
\end{gather*}
Write $(w_1,w_2)=(x_1+\ii y_1,x_2 + \ii y_2)$. Consider the following $S^1$-action $ \rho_t \colon \C^2  \to \C^2$: 
\[  (x_1+\ii y_1,x_2 + \ii y_2)  \mapsto (x_1 + \ii (y_1\cos t+y_2\sin t), x_2 + \ii (-y_1\sin t  + y_2\cos t)),\]
where $t \in S^1 = \R/2\pi\Z$. 
Note that $\rho_t$ preserves $|w_1|^2+|w_2|^2$. Assume $\chi \colon [0,\infty) \to \R$ is an increasing cutoff function with $\chi(r) = 0$ for $r\leq \varepsilon$ and $\chi(r) =\pi$ for $r\geq 2\varepsilon$.

Then the local model of reflection is defined as 
\begin{align*}
    R\colon U & \to U \\
    [0:z_1:z_2] & \mapsto [0:\bar{z}_1:\bar{z}_2], \\
    [1:w_1:w_2] & \mapsto [1:\rho_{\chi(r)}(w_1,w_2)].
\end{align*}
Extending $R$ by the identity over $M-\nu(E)$, we obtain the reflection $r_E\colon M \to M$. 

It is easy to see that $R^2$ is the boundary Dehn twist of $U$. In fact, this boundary Dehn twist is smoothly isotopic to the identity via the isotopy below.
Notice that  
\begin{align*}
    R^2 \colon U & \to U \\
    [0:z_1:z_2] & \mapsto [0:z_1:z_2], \\
    [1:w_1:w_2] & \mapsto [1:\rho_{2\chi(r)}(w_1,w_2)].
\end{align*}
can be isotoped to $\id$ in two steps. First, consider another $S^1$-action $\rho'_t:\mathbb{C}^2\to \mathbb{C}^2$:
\[(w_1, w_2)\mapsto(w_1, e^{it} w_2)\]
where $t \in S^1 = \Z/2\pi\Z$.  Then $[\rho_t]=[\rho'_t]\in \pi_1(\mathrm{SO}(4))$. Thus, $R^2$ is isotopic to 
\begin{align*}
    \phi \colon U & \to U \\
    [0:z_1:z_2] & \mapsto [0:z_1:z_2], \\
    [1:w_1:w_2] & \mapsto [1:w_1: e^{2\chi(r)i}w_2].
\end{align*}
via an isotopy $f_s$ supported in $\{\epsilon\leq r\leq 2\epsilon\}$. Next, note that $\phi$ is isotopic to $\id$ via the following isotopy:
\begin{align*}
    g_s \colon U & \to U \\
    [0:z_1:z_2] & \mapsto [0:z_1: e^{2\pi is}z_2], \\ 
    [1:w_1:w_2] & \mapsto [1:w_1:e^{2is\chi(r)}w_2].
\end{align*}
So we can concatenate these two isotopies to get an isotopy from $\id$ to $r_E^2$:
\[
h_s=\begin{cases}
    g_{2s}, &0\leq s\leq 1/2\\
    f_{2-2s}, &1/2\leq s\leq 1
\end{cases}
\]
Again, we can substitute the formula for $h_s$ into the expression for the clutching function $\tilde{\gamma}_s=h_sr_E h_{s+1/2}^{-1}h_{1/2}r_E^{-1}$. Near the exceptional divisor $E$ ($r\geq 2\epsilon$), the formula of $\tilde{\gamma}$ is rather simple:
\[
\tilde{\gamma}_s([z_0:z_1:z_2])=\begin{cases}
    [z_0:z_1:e^{8\pi is}z_2 ], & 0\leq s\leq 1/2\\
    [z_0:z_1:z_2], &1/2\leq s\leq 1
\end{cases}
\]

\subsection{Proof of Theorem \ref{thm:main thm}\ref{thm:main-2}}
\label{sec:proof}
As in Section \ref{sec:Dehn}, let $M\to \tilde{M}\to S^2$ be the smooth fiber bundle associated to the clutching function $\tilde{\gamma}$. The bundle $\tilde{M}$ has a natural splitting $\tilde{M}=\tilde{M}_1\cup \tilde{M}_2$, where $\nu (E)\to \tilde{M}_1\to S^2$ has clutching function $\tilde{\gamma}$ and $\tilde{M}_2=(M-\nu (E))\times S^2$ is a trivial bundle. Note that $\tilde{\gamma}$ fixes the zero section $E$ setwise and makes two full rotations. So we obtain a subbundle $E \to \tilde E\cong S^2\times S^2 \to S^2$ inside $\tilde M_1$. Let $p=[0:1:0]\in E$ be one of the fixed points. It determines a section of $\tilde{E}$, denoted by $\Sigma_p$. Then in $\tilde{E}$, we have:
\begin{gather*}
    E\cdot E = 0, E\cdot \Sigma_p = 1, \Sigma_p\cdot \Sigma_p=2.
\end{gather*}
\begin{Lemma}\label{lemma: intersection number 2}{\ } 
    \begin{enumerate}
        \item $\tilde{E}\cdot\tilde{E}\cdot\tilde{E}= -2$
        \item $\langle p_1(T(\tilde M_1/S^2)),[\tilde E]\rangle= -2$
    \end{enumerate}
\end{Lemma}
\begin{proof}
    Let $N$ be the normal bundle of $\tilde{E}$ in $\tilde{M}_1$. Then $e(N|_E)=-1$ and $e(N|_{\Sigma_p})=0$. So $e(N)=\PD(2[E]-[\Sigma_p])$. And 
    \[\tilde{E}\cdot\tilde{E}\cdot\tilde{E}=\langle e(N)^2,[\tilde{E}]\rangle=-2.\]
    This is the first statement. 
    
    To prove the second statement, use the decomposition
    \[T(\tilde{M}_1/S^2)|_{\tilde{E}}=T(\tilde{E}/S^2)\oplus N.\]
    Thus 
    \[\langle p_1(T(\tilde M_1/S^2)),[\tilde E]\rangle=\langle e(T(\tilde{E}/S^2))^2, [\tilde{E}]\rangle+\langle e(N)^2, [\tilde{E}]\rangle=-2\]
\end{proof}

\begin{Lemma}\label{lemma: index 2}
    Let $\tilde{\mathfrak{s}}$ be a family $\spinc$ structure on $\tilde{M}$. 
    \begin{enumerate}
        \item  Suppose $\PD(c_1(\tilde{\mathfrak{s}}|_{\tilde{M}_1}))= \pm 3[\tilde{E}]\in H_4(\tilde{M}_1,\partial\tilde{M}_1;\mathbb{Z})$, then 
    $$\langle c_1(\ind_\mathbb{C}(\slashed{D}^+(\tilde{M},\tilde{\mathfrak{s}})), [S^2]\rangle=\pm 1.$$
    \item Suppose $\PD(c_1(\tilde{\mathfrak{s}}|_{\tilde{M}_1}))= \pm [\tilde{E}]\in H_4(\tilde{M}_1,\partial\tilde{M}_1;\mathbb{Z})$, then 
    $$\langle c_1(\ind_\mathbb{C}(\slashed{D}^+(\tilde{M},\tilde{\mathfrak{s}})), [S^2]\rangle=0.$$
    \end{enumerate}
    \end{Lemma}
\begin{proof}
    By the family Atiyah-Singer index theorem, 
    \begin{align*}
        \langle c_1(\ind_\mathbb{C}(\slashed{D}^+(\tilde{M},\tilde{\mathfrak{s}}))),[S^2] \rangle=\frac{1}{48}\langle c_1(\tilde{\mathfrak{s}})^3-c_1(\tilde{\mathfrak{s}}) p_1(T(\tilde{M}/S^2)), [\tilde{M}]\rangle.
    \end{align*}
    Suppose $\PD(c_1(\tilde{\mathfrak{s}}|_{\tilde{M}_1}))=\pm 3[\tilde{E}]$, by Mayer-Vietoris sequence we have 
    \[\PD(c_1(\tilde{\mathfrak{s}}))=\pm 3[\tilde{E}]+\beta\times[S^2]\]
    for some $\beta\in \mathrm{Im}(H_2(M-\nu (E);\mathbb{Z})\to H_2(M;\mathbb{Z}))$. We have:
    \[[\tilde{E}]^i\cdot(\beta\times [S^2])^{3-i}=0,\  \forall 0\leq i<3\]
    and 
    \[\langle p_1(T(\tilde{M}/S^2)),\beta\times [S^2]\rangle=0.\]
    Therefore, 
    \[\langle c_1(\ind_\mathbb{C}(\slashed{D}^+(\tilde{M},\tilde{\mathfrak{s}}))),[S^2]\rangle=\pm \frac{1}{48}(27 [\tilde{E}]^3-3\langle p_1(T(\tilde{M}/S^2)), [\tilde{E}]\rangle)=\pm 1.\]
    Similarly, suppose $\PD(c_1(\tilde{\mathfrak{s}}|_{\tilde{M}_1}))=\pm[\tilde{E}]$,  then  \[\langle c_1(\ind_\mathbb{C}(\slashed{D}^+(\tilde{M},\tilde{\mathfrak{s}}))),[S^2]\rangle=\pm \frac{1}{48}([\tilde{E}]^3-\langle p_1(T(\tilde{M}/S^2)), [\tilde{E}]\rangle)=0.\]
\end{proof}
\begin{Proposition}\label{Prop: nonvanish 2}
    Let $M_0$ be a K3-type 4-manifold. On the blown-up manifold $M=M_0\# n\overline{\mathbb{CP}^2}$, consider the simultaneous reflection $r_E$ along the $n$ exceptional $(-1)$-spheres $E_1,\cdots, E_n$. Let $[\tilde{\gamma}]\in \pi_1(\Diff(M))$ be the standard clutching function associated to $r_E$ constructed above. Let $\mathfrak{s}_0$ be the spin $\spinc$ structure on $M_0$, and $\mathfrak{s}^{(i)}_{k}$ be the $\spinc$ structure on the $i$-th $\overline{\mathbb{CP}^2}$ with first Chern class $\PD(k[E_i])$. Let $\mathfrak{s}=\mathfrak{s}_0\#\mathfrak{s}^{(1)}_{-3}\#\mathfrak{s}^{(2)}_{-1}\#\cdots\#\mathfrak{s}^{(n)}_{-1}$ be a $\spinc$ structure on $M$. Then 
    $\SW([\tilde{\gamma}],\mathfrak{s})=\pm\SW(M_0, \mathfrak{s}_0)$ is odd.
\end{Proposition}
\begin{proof}
Let $M_1$ be the boundary connected sum of $\nu E_1,\cdots, \nu E_n$ in $M$. Then $\partial M_1\cong S^3$ and $\tilde{M}=\tilde{M}_1\cup \tilde{M}_2$, where $M_1\to \tilde{M}_1\to S^2$ has clutching function $\tilde{\gamma}$ while $\tilde{M}_2=(M-M_1)\times S^2$ is a trivial bundle.  Pick any family $\spinc$ structure $\tilde{\mathfrak{s}}$ on $\tilde{M}$ such that $\tilde{\mathfrak{s}}|_M\cong \mathfrak{s}$.  By twisting $\tilde{\mathfrak{s}}$ with a complex line bundle pulled back from $S^2$, we may assume $\tilde{\mathfrak{s}}|_{\tilde{M}_2}$ is the pull-back of $\mathfrak{s}|_{M_2}$. Note that \[d(M, \mathfrak{s}_0\#\mathfrak{s}^{(1)}_{-1}\#\mathfrak{s}^{(2)}_{-1}\#\cdots\#\mathfrak{s}^{(n)}_{-1})=0\]
\[d(M, \mathfrak{s})=d(M, \mathfrak{s}_0\#\mathfrak{s}^{(1)}_{-3}\#\mathfrak{s}^{(2)}_{-1}\#\cdots\#\mathfrak{s}^{(n)}_{-1})=-2\]
and $d(\tilde{M}, \tilde{\mathfrak{s}})=d(M, \mathfrak{s})+2=0$.  Note that the first Chern class of the index bundle is additive under boundary connected sum. Thus, by Proposition \ref{Prop: gluing}, Lemma \ref{lemma: index 2} and the blow-up formula for Seiberg--Witten invariant, we have 
\begin{align*}
    \SW([\tilde{\gamma}],\mathfrak{s})&=-\langle c_1(\ind_\mathbb{C}(\slashed{D}^+(\tilde{M},\tilde{\mathfrak{s}})), [S^2]\rangle \cdot \SW(M,\mathfrak{s}_0\#\mathfrak{s}^{(1)}_{-1}\#\mathfrak{s}^{(2)}_{-1}\#\cdots\#\mathfrak{s}^{(n)}_{-1})\\
    &=\pm \SW(M_0,\mathfrak{s}_0)\equiv 1\mod 2
\end{align*}
\end{proof}

\begin{proof}[Proof of Theorem \ref{thm:main thm}\ref{thm:main-2}]
    Let $\mathfrak{s}=\mathfrak{s}_0\#\mathfrak{s}^{(1)}_{-3}\#\mathfrak{s}^{(2)}_{-1}\#\cdots\#\mathfrak{s}^{(n)}_{-1}$. Then $r_E^*\mathfrak{s}=\mathfrak{s}_0\#\mathfrak{s}^{(1)}_{3}\#\mathfrak{s}^{(2)}_{1}\#\cdots\#\mathfrak{s}^{(n)}_{1}=\bar{\mathfrak{s}}$, so we have 
    \[\SW([\xi],\mathfrak{s})=-\SW([\xi],r_E^*\mathfrak{s})\]
    by Proposition \ref{Prop: charge conjugation}.  Therefore, by Corollary \ref{relation of SW} and Proposition \ref{Prop: nonvanish 2},
    \[\SW([\tilde{\eta}],\mathfrak{s})\equiv\SW([\tilde{\gamma}],\mathfrak{s})\neq 0 \mod 2\]
    So the corresponding family $\tilde M\to S^2$ can't be a trivial bundle.
    By Lemma \ref{lem:criterion}, the theorem holds.
\end{proof}

\section{Non-kinetic homotopy coherent actions}
\label{sec:K3}
Recall that a homotopy coherent action $H\colon BG \to B\Diff(M)$ is said to be \textbf{kinetic}, if there exists a strict action $G\to \Diff(M)$ that induces $H$.

In this section, we will show the existence of non-kinetic homotopy coherent $\Z/2$ action on certain closed $4$-manifolds.
This can be proved by a similar argument in \cite[Section 3]{KPTcoherent}.

\begin{lemma}
    \label{lem:commute}
    Let $\varphi$ be the neck Dehn twist along a smoothly embedded $S^3$ in $M$.
    Then there exists an isotopy $H_t$ from $\varphi^2$ to $\id$ such that $H_t\varphi=\varphi H_t$ for each $t$.
\end{lemma}
\begin{proof}
    This is in fact similar to \cite[Lemma 3.1]{KPTcoherent}. 
    Choose a tubular neighborhood $S^3 \times [0,1] \subset M$ such that the Dehn twist $\varphi$ is supported in $S^3\times [0,1]$, and $\varphi$ is
    \begin{equation}
        \varphi(y,s) = (\alpha(\chi(s)) \cdot y,s),
    \end{equation}
    where $\alpha \colon [0,1] \to SO(4)$ is a loop representing the generator of $\pi_1(SO(4)) = \Z/2$ and $\chi\colon \R \to [0,1]$ is a smooth nondecreasing function with $\chi(s)=0$ for $s\leq 1/6$ and $\chi(s)=1$ for $s\geq 1/3$. We may assume $\alpha(t_1)\alpha(t_2)=\alpha(t_2)\alpha(t_1)$ for any $t_1, t_2\in [0,1]$.
     
    We further shift this Dehn twist to $\varphi'$ which is supported in $S^3\times [0,1]$ by 
    \begin{equation}
        \varphi'(y,s) = (\alpha(\chi(s-1/2)\cdot y, s).
    \end{equation}
    Then $\varphi$ can be isotoped to $\varphi'$ via $H_t'$ as 
    \begin{equation}
        H_t'(y,s) = (\alpha(\chi(s-g(t)))\cdot y,s),
    \end{equation}
    where $g\colon \R \to [0,1/2]$ is a smooth nondecreasing function such that $g(t) = 0$ for $t\leq 1/3$ and $g(t) = 1/2$ for $t \geq 2/3$.
    Also $H_t' \varphi = \varphi H_t'$ for all $t$.

    Notice that $\varphi'^2$ can be smoothly isotoped to $\id$ via an isotopy $H_t''$ which is supported in $S^3 \times [1/2,1]$.
    Hence $H_t''$ commutes with $\varphi$.
    Therefore, the concatenation of $(H_t')^2$ and $H_t''$ yields the desired isotopy $H_t$ from $\varphi^2$ to $\id$, and the commutativity follows.
\end{proof}

\begin{proof}[Proof of Theorem \ref{thm:K3}]
We will use the isotopy above to build a homotopy coherent action. 

Let us first consider the standard simplicial model of $B\Z/2$. 
The set of $n$-simplices $X_n$ is $\{(x_1,\dots,x_n) \mid x_i \in \Z/2 \}$, and the face maps $d_i\colon X_n \to X_{n-1}$ are defined as 
\begin{align*}
    d_0(x_1,\dots,x_n) & = (x_2,x_3,\dots,x_n), \\
    d_i(x_1,\dots,x_n) & = (x_1,\dots, x_i+x_{i+1},\dots, x_n), &  0<i<n \\
    d_n(x_1,\dots,x_n) & = (x_1,\dots,x_{n-1}).
\end{align*}
The only nondegenerate simplex in $X_n$ is $e_n = (1,1,\dots,1)$, with 
\begin{align*}
    d_0(e_n) & = e_{n-1}, \\
    d_n(e_n) & = e_{n-1}, \\
    d_i(e_n) & = (1,\dots, 1,0,1,\dots,1), & 0<i<n.
\end{align*}

We now construct the bundle inductively over the skeleta of $B\mathbb{Z}/2$. 
Starting from the trivial $M$-bundle over the $0$-cell, we glue the bundle over $1$-cell via $\varphi$. 
Next, we can use the isotopy $H_t$ to glue in the $X$-bundle over $2$-cell.

Then we extend the construction inductively over the $n$-cells for $n\geq 3$ .
Suppose there is already a gluing datum $h_{n-1}$ for the $n-1$ cell $e_{n-1}$. Note that only $d_0$ and $d_n$ will map $e_n$ into nondegenerate simplex, we only need to find a gluing datum $h_n$ interpolating between $h_{n-1}\circ \varphi$ and $\varphi \circ h_{n-1}$.

Since $h_2$ is exactly $H_t$ and $H_t$ commutes with $\varphi$, we can simply take $h_3$ to be constant, which also commutes with $\varphi$. Therefore we can simply take constant maps to obtain all gluing data via the commutativity.
And this completes the proof of Theorem \ref{thm:K3}.
\end{proof}

\begin{proof}[Proof of Corollary \ref{cor:nonkinetic}]
    Notice that the monodromy of the homotopy coherent action constructed in Theorem \ref{thm:K3} is $\varphi$.  
    Since $\varphi$ is not smoothly isotopic to $\id$, the homotopy coherent action $H\colon B\Z/2 \to B\Diff(M)$ is not nullhomotopic.
    By \cite[Theorem 1]{Matsumoto}, such $M$ does not admit nontrivial homologically trivial involutions, and this implies such homotopy coherent action is non-kinetic.

    By \cite{KMDehn} and \cite{linstablization}, the neck Dehn twist of $K3\#K3$ and its extension to $K3\#K3\#S^2\times S^2$ are not smoothly isotopic to $\id$. 
    Hence these manifolds indeed admit non-kinetic homotopy coherent actions.
\end{proof}
\begin{rmk}
    In fact, the argument of Kronheimer--Mrowka \cite{KMDehn} and J. Lin \cite{linstablization} also works for $K3$-type $4$-manifolds.
    As a result, there exist non-kinetic homotopy coherent actions on $M_1\#M_2$ and $M_1\#M_2\#S^2\times S^2$ when $M_i$ are of $K3$-type.
\end{rmk}

\bibliographystyle{plain}
\bibliography{ref}
\end{document}